\documentclass[11pt,reqno]{amsart}
\usepackage{amsmath, amsfonts, amsthm, amssymb, color, cite}

\newtheorem{theorem}{Theorem}[section]
\newtheorem{lemma}{Lemma}[section]
\newtheorem{proposition}{Proposition}[section]

\theoremstyle{definition}
\newtheorem{definition}{Definition}[section]

\theoremstyle{remark}

\numberwithin{equation}{section}
\allowdisplaybreaks

\def\va{\varphi}

\newcommand{\R}{{\mathbb R}}

\def\f{\frac}

\def\hf1{^\f{1}{1-\xi^2}}

\newcommand{\eps}{\varepsilon}

\author[Wentao Cao]{Wentao Cao}
\address{Institute f\"{u}r mathematik, Universit\"{a}t Leipzig, D-04109, Leipzig, Germany}
\email{wentao.cao@math.uni-leipzig.de}

\author[Teng Wang]{Teng Wang}
\address{ College of Applied Sciences, Beijing University of Technology, Beijing 100124, China }
\email{tengwang@amss.ac.cn}

\title[Vanishing viscosity limit]
{Vanishing viscosity limit for viscous Burgers-Vlasov equations}

\date{\today}

\begin{document}

\begin{abstract}
We establish the vanishing viscosity limit of viscous Burgers-Vlasov equations for one dimensional kinetic model about interactions between a viscous fluid
and dispersed particles by using compensated compactness technique and the evolution of level sets arguments.
The limit we obtained is exactly a finite-energy weak solution to the inviscid  equations.
\end{abstract}
\maketitle
\medskip

\noindent {\bf 2010 AMS Classification}:  76T10, 35F20, 35Q35, 35Q72,  45K05, 82D05.
\medskip


\section{Introduction}
%

In this note we consider the  vanishing viscosity limit of the following viscous Burgers-Vlasov equations:
\begin{equation}\label{e:vis-burgers-vlasov}
\left\{
\begin{array}{ll}
\displaystyle u_t+uu_x=\eps u_{xx}+\int_\R fvdv-u\int_\R fdv,  \\
\displaystyle f_t+v f_x+(f(u-v))_v=0,
\end{array}
\right.
\end{equation}
with the  initial data
\begin{equation}\label{e:initial}
u(x, 0)=u_0(x), \quad f(x, v, 0)=f_0(x, v),
\end{equation}
such that
\begin{equation}\label{e:uf0-asymptoic}
\lim_{x\rightarrow\pm\infty}u_0(x)=u^{\pm},  \quad \lim_{x, v\rightarrow\pm\infty}f_0(x, v)=0,
\end{equation}
with $u^{\pm}$ being constant states and allowed to be different, here $u(x, t)$ is the bulk velocity of the viscous gas at position $x\in\R$ and time $t\geq0$. $f(x, v, t)$ is the distribution function of the particles occupying at time $t$, the position $x$ with velocity $v\in\R$. $\eps\in (0, \eps_0)$ with some $0<\eps_0<1$ is the viscosity of the gas.

The system \eqref{e:vis-burgers-vlasov} is related to a kinetic model of a two-phase flow in which a dispersed phase interacts with a kind of viscous gas. Such model arises in the description of various combustion phenomena, e.g. diesel engines. The model reads
\begin{equation}\label{e:model}
\left\{
\begin{array}{ll}
\displaystyle \rho_g(u_t+uu_x-\eps u_{xx})=E_d,  \\
\displaystyle f_t+v f_x+(F_df)_v=0,
\end{array}
\right.
\end{equation}
here $\rho_g$ is the density of the gas. The force term $E_d$ describing the exchange of impulse between the gas and the particles has a close relation with the drag force $F_d$ describing the friction of the viscous fluid in the droplets. The relation can be seen from the following formulas:
\begin{equation}\label{e:efd}
\begin{split}
&E_d=\mathcal{C}(r)\rho_p(u_p-u), \quad F_d=\mathcal{C}(r)(u(x, t)-v),\\
&\rho_p=\frac{4\pi}{3}\rho_lr^3\int_\R f(x, v, t)dv,\quad
\rho_p u_p=\frac{4\pi}{3}\rho_lr^3\int_\R f(x, v, t)vdv.
\end{split}
\end{equation}
In \eqref{e:efd},  $\rho_l$ is the density of liquid. $\mathcal{C}(r)$ is a constant depending on the radius $r$ of the droplets. In  \eqref{e:model}, the viscous Burgers' equation, i.e. the first equation models the evolution of viscous gas, while the Vlasov like equation, i.e. the second equation describes the evolution of the dispersed phase. Derivation of the model can be found in \cite{Will85}. Further information about our assumptions on \eqref{e:model} can also be found in \cite{G01, DR99}. We remark that when all the constants in $E_d/\rho_g$ and $F_d$  are all assumed to be 1, then \eqref{e:model} becomes \eqref{e:vis-burgers-vlasov}.

As for the well-posedness results of \eqref{e:vis-burgers-vlasov} or \eqref{e:model}, global existence and uniqueness of classical solutions to the  Cauchy problem with regular initial data  have been considered in \cite{DR99}. Global existence of weak solutions with finite energy is studied in \cite{G01}. Other complicated models of interactions between fluid-kinetic models, such as incompressible/compressible Euler/Navier-Stokes equations coupled with Vlasov/-Fokker-Planck equation are studied in \cite{RS96-1, RS96-2, Y13, MV07} and reference therein. Asymptotic problems like hydrodynamic limit and stratified limit of \eqref{e:model} are also considered in \cite{G01}, and one can see  \cite{BMP07, NPS01, GJ04-1, GJ04-2, Ja00-C, Ja00-P, MV08} for  asymptotic problems of other models related to Vlasov equations.  For the vanishing viscosity limit of Navier-Stokes equations, $L^p$ compensated compactness framework of $2\times2$ system of conservation laws is applied  to  yield the result in \cite{CP10}, which is also used in \cite{GL16}.

Our goal is to show that when $\eps\rightarrow0,$ smooth solutions to \eqref{e:vis-burgers-vlasov}-\eqref{e:initial} converge to a finite-energy weak solution to the following zero-viscosity equations:
\begin{equation}\label{e:burgers-vlasov}
\left\{
\begin{array}{ll}
\displaystyle u_t+uu_x=\int_\R fvdv-u\int_\R fdv,  \\
\displaystyle f_t+v f_x+(f(u-v))_v=0.
\end{array}
\right.
\end{equation}
The relative total energy for \eqref{e:burgers-vlasov} is denoted as
\begin{equation*}
E[u, f]:=\displaystyle\frac{1}{2}\int_\R (u-\bar{u})^2dx+\frac{1}{2}\int_\R\!\int_\R f(1+v^2)dvdx,
\end{equation*}
where  smooth monotone function $\bar{u}(x)$ is constructed as
\begin{equation}\label{e:ubar}
\bar{u}=
\begin{cases}
u^+, \quad &x\geq L_0;\\
\text{monotone}, \quad &-L_0<x<L_0;\\
u^-, \quad &x\leq -L_0
\end{cases}
\end{equation}
with $L_0>0$ large. As we can see from the formula, the relative total energy is the sum of the kinetic energy of the fluid and the particle (in statistic sense). In the paper we denote $[0, \infty)$  as $\R_+$,  $\R^2=\R\times\R.$  Finite-energy weak solutions to \eqref{e:burgers-vlasov} are defined as follows.
\begin{definition}\label{d:energy-solution}
Let $(u_0, f_0)$ be  given initial data with relative finite energy with respect to the end-states $(u^{\pm}, 0)$ at infinity, i.e. $E[u_0, f_0]\leq E_0<\infty.$ For any $T\in\R_+,$ a pair of functions $u:\R\times[0, T]\rightarrow\R,$ $f:\R^2\times[0, T]\rightarrow\R_+$ is called a \textit{finite-energy weak solution} of Cauchy problem \eqref{e:burgers-vlasov} and \eqref{e:initial}-\eqref{e:uf0-asymptoic} if the following holds:
\begin{itemize}
\item[(1)] There is a bounded function $C(E_0, t)$ defined on $\R_+\times[0, T]$, which is continuous in $t$ for each $E\in\R_+$ such that for a.e. $t>0,$
\begin{equation}\label{e:energyineq}
E[u, f](t)+\int_0^t\!\!\!\int\!\!\!\int f(u-v)^2dvdxds\leq C(E_0, t).
\end{equation}
\item[(2)] For any $\phi\in C^1_c(\R\times[0, T)),$
\begin{equation}\label{e:burgersweak}
\int_\R\phi(x, 0)u_0(x)dx+\int_0^T\!\!\!\int_\R\left(u\phi_t+\frac{1}{2}u^2 \phi_x+\phi\int_\R f(v-u)dv\right)dxdt=0,
\end{equation}
and for any $\varphi\in C^1_c((\R^2\times[0, T)),$
\begin{equation}\label{e:vlasovweak}
\int_\R\!\!\int_\R\varphi(x, v, 0)f_0(x, v)dxdv+\int_0^T\!\!\!\int_\R\!\!\int_\R f\varphi_t+ fv\varphi_x+f(u-v)\varphi_vdvdxdt=0.
\end{equation}


\item[(3)]The initial data is achieved in the sense of distributions.
\end{itemize}
\end{definition}
In fact, our general idea on the definition of finite-energy weak solution is that under the condition that the initial total energy is finite the desired solution to \eqref{e:burgers-vlasov} shall also enjoys the finite-energy property, i.e. \eqref{e:energyineq} and satisfies the equations \eqref{e:burgers-vlasov} in weak sense, i.e. \eqref{e:burgersweak} and \eqref{e:vlasovweak}.

Now we are ready to state our main result.
\begin{theorem}\label{t:vanshing-limit}
Let the initial smooth functions $(u_0^\eps, f_0^\eps)$ satisfying the following conditions:
\begin{itemize}
\item[(i)] There exists $E_0>0$ such that $E[u_0^\eps, f_0^\eps]\leq E_0<\infty.$
\item[(ii)] It holds that
 \begin{align*}
 (u_0^\eps(x), f_0^\eps(x, v))\rightarrow(u_0(x), f_0(x, v)) \text{ as }\eps\rightarrow 0
 \end{align*}
 in the sense of distributions with $f_0\geq0.$
\end{itemize}
Let $(u^\eps, f^\eps)$ be the solution to the Cauchy problem \eqref{e:vis-burgers-vlasov} with initial data $(u_0^\eps, f_0^\eps)$ for any fixed $\eps>0.$ Then when $\eps\rightarrow 0$, there exists $(u, f)$ with $u(x, t)\in L^4_{loc}(\R\times[0, T])$ and $ f(x, v, t)\in L^\infty([0, T], (1+v^2)L^1(\R^2))$, which is  a finite-energy weak solution to the Cauchy problem \eqref{e:burgers-vlasov} and \eqref{e:initial}-\eqref{e:uf0-asymptoic} in the sense of Definition \ref{d:energy-solution}, along with corresponding subsequences of $u^\eps$ and $ f^\eps $ (still denoting as $(u^\eps, f^\eps)$) such that
\begin{align*}
&u^\eps\rightarrow u \text{ stongly in } L^r_{loc}(\R\times[0, T]), \text{ for } 1\leq r\leq 4, \text{ as } \eps\rightarrow 0,\\
&f^\eps\rightharpoonup f\text{ weakly in }  L^\infty([0, T], L^1(\R^2)), \text{ as } \eps\rightarrow 0.
\end{align*}

\end{theorem}

Our strategy of proving Theorem \ref{t:vanshing-limit} is to apply $L^p$ compactness framework for scalar conservation laws and study the evolution of level sets after obtaining uniform basic energy estimate and uniform $L^4_{loc}$ estimate. Regarding on the $L^p$ compactness framework for scalar conservation laws,  it is first used in \cite{S82} and then in \cite{Lu92} with some improvement. The compactness framework is also generalized for more models in \cite{YZZ98}. In the present paper, our key difficulty is the estimate of $L^4_{loc}$ boundedness of $u^\eps$, which is obtained by making full use of the flux term of Burgers' equation. To show the $L^1$ weak convergence of $f^\eps$, our technique is studying the evolution of level sets, which is also utilized to handle the convergence of approximate solutions to Vlasov-Possion equations in \cite{ACF17}, but here our novel idea is that estimate the level sets through characteristic map and our key observation is that the Jacobian of the characteristic map is uniformly bounded as time grows.

In the present paper, we denote $\int=\int_\R.$ $C$  is a constant independent of $\eps$ but may vary line to line, $C(\cdot)$ denotes a constant depending on the parameters in the bracket. The rest of the paper is organised as follows. Section \ref{s:uniform} is devoted to show the uniform estimates and the proof of Theorem \ref{t:vanshing-limit} is provided in Section \ref{s:proof}.

\section{Uniform eistimates}\label{s:uniform}
Consider the Cauchy problem \eqref{e:vis-burgers-vlasov} with initial conditions
\begin{align*}
u^\eps(x, 0)=u_0^\eps(x),\quad f^\eps(x, v, 0)=f_0^\eps(x, v)\geq0
\end{align*}
satisfying (i) and (ii) in Theorem \ref{t:vanshing-limit}. When the viscosity $\eps$ is fixed, $u_0^\eps\in C^1(\R)$, $f^\eps_0\in C^1_c(\R^2),$ according to Theorem 2.1 in \cite{DR99}, one is able to obtain the global existence and uniqueness of a smooth solution $(u^\eps, f^\eps)$ with $f^\eps\geq0.$  On the other hand, in Section 4 of \cite{G01}, the author also gained a global weak solution to \eqref{e:vis-burgers-vlasov} when initial data $(u_0^\eps, f^\eps_0)$ only enjoys finite energy property and $f_0^\eps\in (L^1\cap L^2)(\R^2)$. Here, the smooth functions  $(u_0^\eps, f^\eps_0)$  in Theorem \ref{t:vanshing-limit} are regular, and we can also cut off $f_0^\eps$ with smooth function supported in $\{(x, v)||x|+|v|\leq\frac{1}{\eps}\}$ (still denote $f_0^\eps$) to make $f_0^\eps$ compact supported and it still satisfies (i) and (ii). Additionally, the initial data satisfies \eqref{e:uf0-asymptoic},  thus it is not hard to derive that
\begin{equation}\label{e:uf-asymptoic}
\lim_{x\rightarrow\pm\infty}u^\eps(x, t)=u^{\pm},  \quad \lim_{x, v\rightarrow\pm\infty}f^\eps(x, v, t)=0.
\end{equation}
Therefore, for smooth functions $(u_0^\eps, f^\eps_0)$ given in Theorem \ref{t:vanshing-limit} as initial data, for any fixed $\eps>0$, there always exists a unique smooth solution to Cauchy problem \eqref{e:vis-burgers-vlasov} and \eqref{e:initial}-\eqref{e:uf0-asymptoic}  satisfying \eqref{e:uf-asymptoic}.

We now establish two uniform estimates for solutions $(u^\eps, f^\eps)$ with respect to the viscosity coefficient $\eps>0,$  which plays key role in our proof. For simplicity, we drop the upper index $\eps$ in this section.

\subsection{Energy estimate}

With the help of the partial dissipative effect of the source terms in \eqref{e:vis-burgers-vlasov}, for the relative total energy $E[u, f]$, we have the following lemma.
\begin{lemma}\label{l:basic-energy}
Let $E[u_0, f_0]\leq E_0<\infty$ with positive constant $E_0$ independent of $\eps.$ Then there exists a constant $C=C(E_0, t, \bar{u})$ such that
\begin{equation*}
\sup_{\tau\in[0, t]}E[u, f](\tau)+\int_0^t\!\!\!\int\!\!\!\int f(v-u)^2dvdxd\tau+\int_0^t\!\!\!\int\!\eps |u_x|^2 dxd\tau\leq C.
\end{equation*}
\end{lemma}
\begin{proof}
A direct calculation gives
\begin{align*}
\frac{d E}{dt}&=\frac{d}{dt}\int\frac{1}{2}(u-\bar{u})^2dx+\frac{d}{dt}\int\!\!\!\int\frac{1}{2}f(1+v^2)dxdv\\
&=\int(u-\bar{u})u_t dx+\int\!\!\!\int\frac{1}{2}(1+v^2) f_tdxdv.
\end{align*}
Due to
\begin{align*}
u_t=\eps u_{xx}+\int f(v-u)dv-uu_x,
\end{align*}
using \eqref{e:uf-asymptoic} and integration by parts, we have
\begin{align*}
\int(u-\bar{u})u_t dx&=\int(u-\bar{u})\eps u_{xx} dx-\int(u-\bar{u})uu_xdx+\int\!\!\!\int f(u-\bar{u})(v-u)dvdx\\
&=I_1+I_2+I_3,
\end{align*}
where
\begin{align*}
I_1&=-\!\!\int\eps(u_x-\bar{u}_x)u_x dx,\\
I_2&=-\!\!\int(u-\bar{u})uu_x,\\
I_3&=\int\!\!\int f(u-\bar{u})(v-u)dvdx.
\end{align*}
We then bound the three terms $I_i, i=1, 2, 3$ one by one.
Note that from \eqref{e:ubar}, one is easy to see that $\bar{u}_x$ is bounded and compact supported in $[-L_0, L_0],$  thus we have
\begin{equation}\label{e:i1}
I_1\leq-\int\eps |u_x|^2dx+\eps\int_{-L_0}^{L_0}|\bar{u}_x||u_x|dx\leq-\frac{1}{2}\int\eps |u_x|^2dx+C.
\end{equation}
Similarly, for $I_2$, we have
\begin{equation}\label{e:i2}
\begin{split}
I_2&=-\int[(u-\bar{u})^2(u-\bar{u})_x+\bar{u}_x(u-\bar{u})^2+(u-\bar{u})\bar{u}(u-\bar{u})_x+(u-\bar{u})\bar{u}\bar{u}_x]dx\\
&=-\int[(u-\bar{u})^2(u-\bar{u})_x+\frac{1}{2}\bar{u}_x(u-\bar{u})^2+(u-\bar{u})\bar{u}\bar{u}_x]dx\\
&\leq C\int|u-\bar{u}|^2dx+C\\
&\leq CE+C,
\end{split}
\end{equation}
and
\begin{equation}\label{e:i3}
\begin{split}
I_3&=\int\!\!\!\int fu(v-u)dvdx-\int\!\!\!\int \bar{u}f(v-u)dvdx\\
&\leq\int\!\!\!\int fu(v-u)dvdx+\frac{1}{2}\int\!\!\!\int f(v-u)^2dvdx+C\int\!\!\!\int fdvdx.
\end{split}
\end{equation}
Furthermore, using
$$f_t=-(fv)_x-(f(u-v))_v$$
and \eqref{e:uf-asymptoic} we gain
\begin{equation}\label{e:f1}
\begin{split}
\int\!\!\!\int\frac{1}{2}(1+v^2) f_tdxdv&=-\int\!\!\!\int\frac{1}{2}(1+v^2) [(fv)_x+(f(u-v))_v]dxdv\\
&=\int\!\!\!\int fv(u-v)dvdx.
\end{split}
\end{equation}
Putting \eqref{e:i1}, \eqref{e:i2}, \eqref{e:i3} and \eqref{e:f1} together gives
\begin{align*}
\frac{d E}{dt}\leq CE+C-\frac{1}{2}\int\eps |u_x|^2dx -\frac{1}{2}\int\!\!\!\int f(u-v)^2dvdx.
\end{align*}
Then directly application of Gronwall's inequality contributes to the lemma.
\end{proof}

\subsection{Higher integrability of velocity}
Although by Lemma \ref{l:basic-energy}, one has $u\in L^2,$ we require much higher regularity of $u$. Taking value of the flux term in Burgers' equation, we improve the regularity of $u$ to be $L^4_{loc}$.
\begin{lemma}\label{l:u-l4local}
Let $E[u_0, f_0]\leq E_0<\infty$ with positive constant $E_0$ independent of $\eps.$ Then for any compact set $K\subset\R$ and all $t>0$, there exists a constant $C=C(E_0, K, \bar{u}, t),$ independent of $\eps$ such that
\begin{align*}
\int_0^t\int_K u^4dxd\tau\leq C.
\end{align*}
\end{lemma}
\begin{proof}
 Let $\va$ be an arbitrary smooth compactly supported function such that $\va|_K=1$ and $0\leq\va\leq1.$ Motivated by the proof of Lemma 3.3 in \cite{CP10}, multiplying the viscous Burgers' equation by $\va$ and then integrating with respect to space variable over $(-\infty, x)$, we gain
\begin{align*}
\frac12u^2\va=&\eps u_x\va-\left(\int_{-\infty}^x u\va dy\right)_t+\int_{-\infty}^x
\left(\frac{1}{2}u^2\va_x-\eps u_x\va_x\right)dy\\
&+\int_{-\infty}^x\!\!\va\int f(v-u)dvdy.
\end{align*}
Multiply the above equation by $u^2\va$ and use the viscous Burgers' equation to get
\begin{align*}
\frac12u^4\va^2&=\eps u^2u_x\va^2-\left(u^2\va\int_{-\infty}^xu\va dy\right)_t+2\eps uu_{xx}\va\int_{-\infty}^xu\va dy\\
&\quad-2u^2u_x\va\int_{-\infty}^xu\va dy+2u\va\int f(v-u)dv\int_{-\infty}^xu\va dy\\
&\quad +u^2\va\int_{-\infty}^x
\left(\frac{1}{2}u^2\va_x-\eps u_x\va_x\right)dy+u^2\va\int_{-\infty}^x\!\!\va\int f(v-u)dvdy.
\end{align*}
Integrating over $\R\times(0, t)$ gives

\begin{align*}
\frac12\int_0^t\!\!\int u^4\va^2dxd\tau=\sum_{i=1}^5J_i,
\end{align*}
with
\begin{align*}
J_1&=\int_0^t\!\!\int\left[\eps u^2u_x\va^2
+2\eps uu_{xx}\va\int_{-\infty}^xu\va dy-u^2\va\int_{-\infty}^x\eps u_x\va_xdy\right]dxd\tau,\\
J_2&=\int\left(u_0^2\va\int_{-\infty}^xu_0\va dy\right)dx-\int\left(u^2\va\int_{-\infty}^xu\va dy\right)dx,\\
J_3&=\int_0^t\!\!\int\left[\frac{1}{2}u^2\va\int_{-\infty}^xu^2\va_xdy
-2u^2u_x\va\int_{-\infty}^xu\va dy\right]dxd\tau,\\
J_4&=\int_0^t\!\!\int\left[2u\va\int f(v-u)dv\int_{-\infty}^xu\va dy\right]dxd\tau,\\
J_5&=\int_0^t\!\!\int\left(u^2\va\int_{-\infty}^x\!\!\va\int f(v-u)dvdy\right)dxd\tau.
\end{align*}
In the following, we will estimate $J_i, i=1,\cdots, 5$ one by one.
For $J_1, $ an application of integration by parts yields
\begin{align*}
J_1=&\int_0^t\!\!\int\eps u^2u_x\va^2dxd\tau-\int_0^t\!\!\int\left(
u^2\va\int_{-\infty}^x\eps u_x\va_xdy\right)dxd\tau\\
&-\int_0^t\!\!\int\left(2\eps u_x^2\va\int_{-\infty}^xu\va dy+2\eps u^2u_x\va^2+2\eps uu_x\va_x\int_{-\infty}^xu\va dy\right)dxd\tau,\\
\leq&\delta\int_0^t\!\!\int u^4\va^2 dxd\tau+C(\delta, E_0, K, \bar{u}, t)\int_0^t\!\!\int\eps |u_x|^2 dxd\tau\\
\leq&\delta\int_0^t\!\!\int u^4\va^2 dxd\tau+C(\delta, E_0, K,  \bar{u}, t)
\end{align*}
with small $\delta$ to be determined, where we have used H\"older's inequality, Lemma \ref{l:basic-energy}  and
\begin{align*}
\left|\int_{\infty}^xu\va dy\right|\leq\int (u-\bar{u})^2dx+C(K, \bar{u})\leq C(E_0, K, \bar{u}).
\end{align*}
Again applying basic energy estimate and H\"older's inequality to $J_2$ and $J_3$, we have
\begin{align*}
J_2+J_3\leq \delta\int_0^t\!\!\int u^4\va^2 dxd\tau+C(\delta, E_0, K, \bar{u}, t).
\end{align*}
For $J_4$, the following inequality
\begin{align*}
\int fu^2dv\leq2\int f(v-u)^2 dv+2\int fv^2dv
\end{align*}
implies
\begin{align*}
\int_0^t\!\!\int\!\!\int fu^2 dvdxd\tau\leq C(E_0, \bar{u}, t).
\end{align*}
Thus we have
\begin{align*}
\int_0^t\!\!\int\left|u\int f(v-u)dv\right|dxd\tau
&\leq\int_0^t\!\!\int\left(\int fu^2dv\right)^{\frac{1}{2}}\left(\int f(v-u)^2 dv\right)^{\frac{1}{2}}dxd\tau\\
&\leq  C(E_0, \bar{u},  t),
\end{align*}
which then gives $J_4\leq C(E_0,  K, \bar{u}, t).$
Similarly, we can derive
\begin{align*}
\left|\int f(v-u)dv\right|\leq \left(\int fdv\right)^{\frac{1}{2}}\left(\int f(v-u)^2 dv\right)^{\frac{1}{2}}.
\end{align*}
Hence we gain $J_5\leq C(E_0, K, \bar{u}, t).$  Collecting all the estimates of $J_i, i=1, \cdots, 5$ and taking $\delta\leq\frac{1}{16} $ yield the lemma.
\end{proof}

\section{Vanishing viscosity limit}\label{s:proof}

In this section, we will use the estimates in Section \ref{s:uniform} to establish the convergence of $(u^\eps, f^\eps),$ whose limit is just a finite-energy weak solution to Cauchy problem \eqref{e:burgers-vlasov} and \eqref{e:initial}-\eqref{e:uf0-asymptoic}. Based on the uniform estimates Section \ref{s:uniform}, we get the following:
\begin{align}
&\sup_{\tau\in[0, t]}E[u^\eps, f^\eps](\tau)\leq C, \label{eq:total-energy} \\
&\int_0^t\!\!\!\int\!\!\!\int f^\eps(v-u^\eps)^2dvdxd\tau,\label{eq:fvu-bound}\\
&\int_0^t\!\!\!\int\!\eps |u^\eps_x|^2 dxd\tau\leq C, \label{eq:u-vis-bound}\\
&\int_0^t\int_K (u^\eps)^4dxd\tau\leq C, \text{ for any compact set } K\subset\R. \label{eq:u-l4-local-bound}
\end{align}
We then divide the proof into three subsections. In Section \ref{s:limit-f} and Section \ref{s:limit-u}, we will apply the uniform estimates \eqref{eq:u-vis-bound}, \eqref{eq:u-l4-local-bound}  and \ref{eq:total-energy} to show the convergence of $f^\eps$ and $u^\eps$ respectively. In Section \ref{s:limit-equation} we will prove the obtained limit is our desired solution.

\subsection{Limit of  distribution function}\label{s:limit-f}
To show $f^\eps$ is weakly compact in $L^1(\R^2),$ a.e. $t\in[0, T],$ one can study the  evolution of  level sets of $f^\eps(x, v, t)$ and $f_0^\eps(x, v),$ which is motivated by Steps 1-3 in the proof of Theorem 2.7 in \cite{ACF17 }. For our case, the estimate on the level sets is done through the characteristic map. Our key observation is that the Jacobian of the characteristic map remains uniformly bounded as time grows.

Assume $|\{f_0=k\}|=0$ for every $k\in\mathbb{N}.$ (Otherwise one can consider $\tau+k$ in place of $k$  for some $\tau\in(0, 1).$)  From the strong convergence of $f_0^\eps,$ one could deduce that when $\eps\rightarrow0,$
\begin{align*}
f_0^{\eps, k}=:\mathbf{1}_{\{k\leq f_0^\eps<k+1\}}f_0^\eps\rightarrow f_0^k:=\mathbf{1}_{\{k\leq f_0<k+1\}}f_0 \text{ in } L^1(\R^2) \text{ for any } k\in\mathbb{N},
\end{align*}
where $\mathbf{1}_A$ is the characteristic function of set $A.$ We shall also analyze the evolution of corresponding level sets of $f^\eps(x, v, t)$.  In fact, for Vlasov equation
$$f^\eps_{t}+(f^\eps v)_{x}+(f^\eps (u^\eps-v))_{v}=0$$
with $u^\eps$ being a smooth function,  the equation can be rewritten as
\begin{equation*}
f^\eps_{t}+vf^\eps_{x}+(u^\eps-v)f^\eps_{v}=f^\eps,  \quad  f^\eps(x, v, 0)=f^\eps_{0}(x,v),
\end{equation*}
which has a unique smooth solution
\begin{equation}\label{e:f}
\displaystyle f^\eps(x,v,t)=f^\eps_{0}(X^\eps(0; x,v,t), V^\eps(0; x,v,t))e^t,
\end{equation}
where  $X(s;x,v,t), V(t;x,v,t)$ are backward characteristic curves satisfying
\begin{align}
&\frac{dX^\eps(s;x,v,t)}{ds}=V^\eps(s;x,v,t), \quad  X^\eps(t;x,v,t)=x; \label{e:x} \\
&\frac{dV^\eps(s;x,v,t)}{ds}=u^\eps(s,X^\eps(s;x,v,t))-V^\eps(s;x,v,t), \quad  V^\eps(t;x,v,t)=v. \label{e:v}
\end{align}
It is not hard to see that $X^\eps, V^\eps$ are well-defined from the theory of ordinary differential equation (ODE). From \eqref{e:x} and \eqref{e:v}, one is able to show that the Jacobian $J(t)=\det \nabla_{x, v}(X^\eps, V^\eps)$ of the map $\mathcal{J}(s):(x, v)\mapsto(X^\eps, V^\eps)$ is nonnegative and satisfies the following ODE
\begin{equation*}
\left\{
\begin{array}{ll}
\displaystyle \frac{d J(t)}{dt}=J(t)\text{div}_{x, v}(X^\eps, V^\eps)=-J(t),\\
\displaystyle J(t)=1,
\end{array}
\right.
\end{equation*}
so $J(\tau)=e^{t-\tau}$ for any $\tau\in [0, t].$  For any $t>0,$ noting \eqref{e:f}, set
\begin{align*}
f^{\eps, k}(x, v, t)= e^t\mathbf{1}_{\{k\leq f^\eps_0 \circ \mathcal{J}(0)<k+1\}}f^\eps_0\circ\mathcal{J}(0).
\end{align*}
Then we have $f^{\eps, k}$ is a weak solution to the Vlasov equation and
\begin{align*}
\int\!\!\int f^{\eps, k}dvdx=&\int\!\!\int e^t\mathbf{1}_{\{k\leq f^\eps_0 \circ \mathcal{J}(0)<k+1\}}f^\eps_0\circ\mathcal{J}(0) dvdx\\
=&\int\!\!\int e^t\mathbf{1}_{\{k\leq f^\eps_0<k+1\}}f^\eps_0 J(0)^{-1}dV^\eps dX^\eps\\
=&\int\!\!\int f^{\eps, k}_0 dV^\eps dX^\eps,
\end{align*}
hence for any $t>0, $
$$\|f^{\eps, k}(x, v, t)\|_{L^1(\R^2)}=\|f^{\eps, k}_0(x, v)\|_{L^1(\R^2)}.$$
It is easy to find that $0\leq f^{\eps, k}\leq (k+1)e^T, $ thus up to subsequences, for any $k\in\mathbb{N}.$
\begin{align*}
f^{\eps, k}\rightharpoonup f^k \text{ weakly* in  } L^{\infty}(\R^2\times[0, T]), \text{ as } \eps\rightarrow0.
\end{align*}
Similar to \cite{ACF17 }, one can use the test function $\phi(t)\mathbf{1}_K\text{sign}(f^k)(x, v, t)$ for any compact subset $K\subset\R^2$ and any $\phi\in C_c^\infty(\R_+)$ in the above weak convergence to show
$$\|f^{k}(x, v, t)\|_{L^1(\R^2)}\leq\|f^{k}_0(x, v)\|_{L^1(\R^2)}$$
for almost all $t$. We define
\begin{align*}
f(x, v, t):=\sum_{k=0}^\infty f^k(x, v, t), \text{ for } (x, v, t)\in\R^2\times[0, T],
\end{align*}
then it is easy to derive
$$\|f\|_{L^1(\R^2)}\leq\sum_{k=0}^\infty \|f^{k}(x, v, t)\|_{L^1(\R^2)}\leq \sum_{k=0}^\infty\|f^{k}_0(x, v)\|_{L^1(\R^2)}=\|f_0\|_{L^1(\R^2)}.$$
We then prove
\begin{equation}\label{e:f-convergent}
f^{\eps}\rightharpoonup f \text{ weakly in  } L^{\infty}([0, T], L^1(\R^2)), \text{ as } \eps\rightarrow0.
\end{equation}
In fact, for any $\va\in L^{\infty}(\R^2),$ we derive
\begin{align*}
&\left|\int\!\!\int \va(f^\eps-f)dxdv\right|=\left|\sum_{k=0}^\infty\int\!\!\int \va(f^{\eps, k}-f^k)dxdv\right|\\
\leq&\left|\sum_{k=0}^{k_0-1}\int\!\!\int \va(f^{\eps, k}-f^k)dxdv\right|+\sum_{k=k_0}^\infty\int\!\!\int |\va||f^{\eps, k}|dxdv+
\sum_{k=k_0}^\infty\int\!\!\int |\va||f^k|dxdv.
\end{align*}
The first term converges to zero as $\eps\rightarrow0$ due to the weak convergence of $f^{\eps, k}$ for any finite $k_0.$ The last two terms can be estimated by
\begin{align*}
&\sum_{k=k_0}^\infty\int\!\!\int |\va||f^{\eps, k}|dxdv+
\sum_{k=k_0}^\infty\int\!\!\int |\va||f^k|dxdv\\
\leq &\|\va\|_{L^\infty(\R^2)}\left(\|f_0^\eps\mathbf{1}_{\{f_0^\eps\geq k_0\}}\|_{L^1(\R^2)}+\|f_0\mathbf{1}_{\{f_0\geq k_0\}}\|_{L^1(\R^2)}\right),
\end{align*}
which converges to zero as $k_0\rightarrow\infty$ thanks to the fact that  $f^{\eps}_0$ and $f_0$ are bounded in $L^1(\R^2)$.
Finally, we obtain \eqref{e:f-convergent}.

\subsection{Limit of velocity}\label{s:limit-u}
To show the convergence of $u^\eps$, we will utilize the $L^p$ compactness framework for Burgers' equation. Thus we first recall
a proposition on such framework, which is resulting from \cite{Lu92}. The framework is shown mainly by div-curl lemma and compactness of some entropies for Burgers' equation.
\begin{proposition}\label{p:framework}
Let $u^\eps(x, t)$ satisfy the following two conditions:
\begin{itemize}
\item[(C1)] $u^\eps(x, t)$ is uniformly bounded in $L^p_{loc}(\R\times[0, T])$ for some $p>2$;
\item[(C2)] Both $\partial_t I_n(u^\eps(x, t))+\partial_x F_n(u^\eps(x, t))$ and  $\partial_t F_n(u^\eps(x, t))+\partial_x\Phi_n(u^\eps(x, t))$ lie in a compact set of $H^{-1}_{loc}(\R\times[0, T])$ with respect to $\eps$ for any $n\in\mathbb{N},$ where
\begin{align*}
&I_n(u)=\begin{cases}u, &\text{ when } |u|\leq n,\\
0, &\text{ when } |u|\geq 2n,
\end{cases}
\text{ and } I_n\in C^2(\R), \quad |I_n(u)|\leq |u|, \quad |I_n'(u)|\leq 2,\\
&F_n(u)=\int_0^u I_n'(s)sds, \quad \Phi_n(u)=\int_0^u F'_n(s)sds.
\end{align*}
\end{itemize}
Then there exists a subsequence (still denoted $u^\eps$) such that $u^\eps\rightarrow u$ almost everywhere and strongly in $L^r_{loc}(\R\times[0, T])$ for all $1\leq r\leq p.$
\end{proposition}
With such $L^p$ framework, we only need to verify (C1)-(C2) to show the convergence of $u^\eps.$ It is easy to see from \eqref{eq:u-l4-local-bound}, we have
$$ u^\eps\in L^4_{loc}(\R\times[0, T]),$$
thus (C1) is satisfied by $u^\eps(x, t)$ for $p=4.$

To verify (C2), we also require an important lemma: Murat's lemma, which is useful in proving compactness of some sequences.
\begin{lemma}\label{l:murat} (Murat's Lemma \cite{Chen86, Tar79})
Let $\Omega\in\mathbb{R}^n$ be a open bounded subset, $D_1$ be a compact
set in $W^{-1,a}_{loc}(\Omega) $, $D_2$ be a bounded set in $W_{loc}^{-1,b}(\Omega)$ for some constants $a, b$ satisfying
$1<a\leq2<b.$. Furthermore, let $D_0\subset\mathcal{D}(\Omega)$ such that $D_0\subset D_1\cap D_2.$  Then there exists $D_*$, a compact set in $H^{-1}_{loc}(\Omega)$ such
that $D_0\subset D_* $.
\end{lemma}
Then we turn to verify (C2). Noting that both  $I_n(u) $ and $F_n(u)$ are $C^2$ compact supported functions of $u,$  one gets
\begin{equation}
\label{e:w-1-bound}
\left.\begin{array}{ll}
\partial_t I_n(u^\eps(x, t))+\partial_x F_n(u^\eps(x, t)),\\
\partial_t F_n(u^\eps(x, t))+\partial_x\Phi_n(u^\eps(x, t))
\end{array}
\right\}\text{ are bounded in } W^{-1, \infty}_{loc}(\R\times[0, T]).
\end{equation}
Furthermore, from
\begin{align*}
\partial_t F_n(u^\eps)+\partial_x\Phi_n(u^\eps)&=\eps u^\eps_{xx} F_n'(u^\eps)+F_n'(u^\eps)\int\!\! f^\eps(v-u^\eps)dv\\
&=(\eps u^\eps_xF'_n(u^\eps))_x-\eps (u^\eps_x)^2 F''_n(u^\eps)+F_n'(u^\eps)\int\!\! f^\eps(v-u^\eps)dv,
\end{align*}
along with $|F_n'(u^\eps)|\leq |I'_n(u^\eps)u^\eps|\leq C(n),$  and $|F_n''(u^\eps)|\leq C(n),$  one can derive
\begin{equation}
\label{e:w-1-alpha-1}
\partial_t F_n(u^\eps)+\partial_x\Phi_n(u^\eps) \text{ is compact in } W^{-1, \alpha}_{loc}(\R\times[0, T]) \text{ for some } \alpha\in(1, 2).
\end{equation}
Indeed, using $\sqrt{\eps}u_x^\eps\in L^2(\R\times[0, T])$  and $|F_n'(u^\eps)|\leq C(n),$ we get $(\eps u^\eps_xF'_n(u^\eps))_x$ is compact in $H^{-1}_{loc}(\R\times[0, T]).$ Employing Lemma \ref{l:basic-energy} we also have
$$-\eps (u^\eps_x)^2 F''_n(u^\eps)+F_n'(u^\eps)\int\!\! f^\eps(v-u^\eps)dv\in L^1(\R\times[0, T]),$$
 which also implies that it is compact in
$W^{-1, \alpha}_{loc}(\R\times[0, T])$ for some $\alpha\in(1, 2)$ by embedding theorem and Schauder theorem, so we obtain \eqref{e:w-1-alpha-1}.  Similarly, one can also gain
\begin{equation}
\label{e:w-1-alpha-2}
\partial_t I_n(u^\eps)+\partial_x F_n(u^\eps)\text{ is compact in } W^{-1, \alpha}_{loc}(\R\times[0, T]) \text{ for some } \alpha\in(1, 2).
\end{equation}
Combining with \eqref{e:w-1-bound}, \eqref{e:w-1-alpha-1} and \eqref{e:w-1-alpha-2}, applying Murat's Lemma (see Lemma \ref{l:murat}), one gets (C2). Therefore, applying Proposition \ref{p:framework} to $u^\eps,$  one can seek a $u(x, t)\in L^4_{loc}(\R\times[0, T])$ and a subsequence of $u^\eps$ (still denoted as $u^\eps$) such that
\begin{align*}
&u^\eps\rightarrow u \text{ a.e. } \R\times[0, T],\\
&u^\eps\rightarrow u \text{ stongly in } L^r_{loc}(\R\times[0, T]), \text{ as } \eps\rightarrow 0, \text{ for } 1\leq r\leq4.
\end{align*}

\subsection{Limit of equations and conclusions}\label{s:limit-equation}
Before taking limit of equations and proving that $(u, f)$ is a weak solution of the Cauchy problem  \eqref{e:burgers-vlasov} and \eqref{e:initial}-\eqref{e:uf0-asymptoic}, we shall  first get  the convergences of nonlinear terms in the equations: $\int f^\eps dv$, $\int f^\eps vdv$ and $u^\eps\int f^\eps dv.$

\textit{(1) Convergence $~~~~\displaystyle\!\!\int f^\eps vdv\rightarrow\int fvdv$.}   From energy estimate and the lower  semi-continuity of kinetic energy, we deduce that
\begin{align}\label{e:fvenergy}
\int\!\!\int v^2f dxdv\leq\liminf_{\eps\rightarrow0}\int\!\!\int v^2f^\eps dxdv\leq C.
\end{align}
Let $\chi(s)=\mathbf{1}_{[-1, 1]}(s).$  For any $\va\in C_c^\infty(\R),$ for any  $L>0$,  inspired by \cite{ACF17},  we observe that
\begin{align*}
\int\left(\int f^\eps vdv-\int fvdv\right)\va dx=&\int\!\!\int (f^\eps-f)v\chi(\frac{v}{L})\va dvdx\\
&+\int\!\!\int (f^\eps v(1-\chi(\frac{v}{L}))\va dvdx\\
&-\int\!\!\int fv(1-\chi(\frac{v}{L}))\va dvdx.
\end{align*}
The first term in the right hand side converges to 0 due to the weak convergence of $f^\eps$ to $f$ in $L^1(\R^2).$ Thanks to \eqref{e:fvenergy}, the remained two terms can be estimated  as
\begin{align*}
\left|\int\!\!\int (f^\eps v(1-\chi(\frac{v}{L}))\va dvdx\right|\leq\frac{\|\va\|_{L^\infty}}{L}\int\!\!\int f^\eps v^2 dvdx\leq \frac{C\|\va\|_{L^\infty}}{L},\\
\left|\int\!\!\int (fv(1-\chi(\frac{v}{L}))\va dvdx\right|\leq\frac{\|\va\|_{L^\infty}}{L}\int\!\!\int f v^2 dvdx\leq \frac{C\|\va\|_{L^\infty}}{L}.
\end{align*}
Letting $L$ go to infinity, we can get
\begin{align*}
\int f^\eps vdv\rightharpoonup \int fvdv  \text{ weakly in } L^\infty([0, T], L^1(\R)).
\end{align*}

\textit{(2) Convergence  $~~\!\!\int f^\eps dv\rightarrow \!\!\int fdv$.}  Similar to Step 1, we can get
\begin{align*}
\int f^\eps dv\rightharpoonup \int fdv  \text{ weakly in } L^\infty([0, T], L^1(\R)).
\end{align*}

\textit{(3) Convergence  $~~\displaystyle u^\eps\!\!\int f^\eps dv\rightarrow u\!\!\int fdv$.}  We easily get from the energy estimate that for any $T\in\R_+,$
\begin{align*}
\int_0^T\!\!\int (u^\eps)^2 \!\!\int f^\eps dxdv\leq 2\int_0^T\!\!\int\!\!\int v^2f^\eps dxdv+2\int_0^T\!\!\int\!\!\int (u^\eps-v)^2 f^\eps dxdv\leq C.
\end{align*}
Due to the semicontinuity  of the integral functional and the strong compactness of $u^\eps$ and weak compactness of $~~\displaystyle\!\!\int f^\eps dv,$ one also gets
\begin{align}\label{e:fuenergy}
\int_0^T\!\!\int u^2\!\!\int f dxdv\leq\liminf_{\eps\rightarrow0}\int_0^T\!\!\int (u^\eps)^2\!\!\int  f^\eps dxdv\leq C.
\end{align}
Following  the same strategy of proving the convergence of $~~\displaystyle\!\!\int vf^\eps dv,$  denoting $\chi(s)=\mathbf{1}_{[-1, 1]}(s),$  we gain for any $\va\in C^\infty_c(\R\times(0, T)),$ for arbitrary $L>0,$
\begin{align*}
&\int_0^T\!\!\!\int\left(u^\eps\int f^\eps dv-u\int fdv\right)\va dxdt\\
=&\int_0^T\!\!\int\!\!\int \left(u^\eps\chi(\frac{u^\eps}{L})\int f^\eps dv-u\chi(\frac{u}{L})\int fdv\right)\va dvdx\\
&+\int_0^T\!\!\int\!\!\int u^\eps f^\eps (1-\chi(\frac{u^\eps}{L})\va dvdx\\
&-\int_0^T\!\!\int\!\!\int fu(1-\chi(\frac{u}{L})\va dvdx.
\end{align*}
For the first term in the right hand side of the above integrals, since we have uniform bounds on the integrands, by Dominated Convergence Theorem it converges to 0 as $\eps\rightarrow0.$    The last two terms  can be estimated through \eqref{e:fuenergy} as before,
\begin{align*}
\left|\int_0^T\!\!\!\int\!\!\int f^\eps u^\eps(1-\chi(\frac{u^\eps}{L})\va dvdx\right|\leq\frac{\|\va\|_{L^\infty}}{L}\int_0^T\!\!\!\int\!\!\int f^\eps (u^\eps)^2 dvdx\leq \frac{C\|\va\|_{L^\infty}}{L},\\
\left|\int_0^T\!\!\!\int\!\!\int fu(1-\chi(\frac{u}{L})\va dvdx\right|\leq\frac{\|\va\|_{L^\infty}}{L}\int_0^T\!\!\!\int\!\!\int f u^2 dvdx\leq \frac{C\|\va\|_{L^\infty}}{L},
\end{align*}
both of which go to 0 upon letting  $L\rightarrow\infty.$

Now we are in a position to show $(u, f)$ is a finite-energy weak solution to Cauchy problem \eqref{e:burgers-vlasov} and \eqref{e:initial}-\eqref{e:uf0-asymptoic}.

\textit{(1) Weak solutions.} It suffices to show \eqref{e:burgersweak} and \eqref{e:vlasovweak} hold for $(u, f).$ Here we only show \eqref{e:burgersweak} by the uniform estimate  \eqref{eq:u-vis-bound}, since \eqref{e:vlasovweak} can be verified similarly.  Multiplying the first equation in \eqref{e:vis-burgers-vlasov} by $\phi\in C^\infty_c(\R\times[0, T))$, integrating over $\R\times[0, T]$, and employing integration by parts we obtain
\begin{align*}
&\int_\R\phi(x, 0)u^\eps_0(x)dx+\int_0^T\!\!\!\int_\R\left(u^\eps\phi_t+\frac{1}{2}(u^\eps)^2 \phi_x+\phi\int_\R f^\eps(v-u^\eps)dv\right)dx\\
&\qquad -\eps\int_0^T\!\!\!\int u^\eps_x\phi_x dxdt=0,
\end{align*}
For the last term in the left hand side, it follows from \eqref{eq:u-vis-bound} that
\begin{align*}
\left|\eps\int_0^T\!\!\!\int u^\eps_x\phi_x dxdt\right|\leq\sqrt{\eps}\left(\int_0^T\!\!\!\int\eps (u^\eps_x)^2dxdt\right)^{\frac{1}{2}}\left(\int_0^T\!\!\!\int\phi_x^2dxdt\right)^{\frac{1}{2}}\leq C\sqrt{\eps},
\end{align*}
which goes to 0 by letting $\eps\rightarrow0.$  The convergence of the other terms have already obtained, thus we have \eqref{e:burgersweak}. Hence, the obtained limit  $(u, f)$ is a weak solution to \eqref{e:burgers-vlasov}.

\textit{(2) Finite energy.} We also show the obtained limit functions also enjoy the finite energy property. Obviously, from  almost everywhere convergence of $u^\eps$ to $u$ and Lemma \ref{l:basic-energy}, we have
$$u^\eps-\bar{u}\rightarrow u-\bar{u} \text{ strongly in } L^2(\R).$$
Hence by the convexity of energy we gain
\begin{align*}
\int(u-\bar{u})^2dx\leq\liminf_{\eps\rightarrow0}\int(u^\eps-\bar{u})^2dx\leq C.
\end{align*}
Besides,
\begin{align*}
\int_0^T\!\!\!\int\!\!\!\int f(u-v)^2dvdxdt\leq 2\int_0^T\!\!\!\int\!\!\!\int fu^2dvdxdt+2\int_0^T\!\!\!\int\!\!\!\int fv^2dvdxdt\leq C.
\end{align*}
Combing with \eqref{e:fvenergy}, we gain \eqref{e:energyineq}.

Therefore, $(u, f)$ is a finite-energy weak solution to Cauchy problem \eqref{e:burgers-vlasov} and \eqref{e:initial}-\eqref{e:uf0-asymptoic}. Besides, it is easy to verify (3) in Definition \ref{d:energy-solution} for $(u, f)$ by the smoothness and compactness conditin (ii) of $(u_0^\eps, f_0^\eps)$ in Theorem \ref{t:vanshing-limit}. The proof of Theorem \ref{t:vanshing-limit} is then completed.

\section*{Acknowledgments}
The research of the W. Cao is supported by ERC Grant Agreement No. 724298 and
he also thanks for the hospitality of Max-Plank Institute for Mathematics in the Sciences.
The work of T. Wang is partially supported by the NNSFC  grant No. 11601031.

\end{document}